\documentclass[letterpaper,final,twoside,reqno]{amsart}
\usepackage[all]{xy}
\usepackage{xspace}
\usepackage{hyperref}

\newcommand{\baseRing}[1]{\ensuremath{\mathbb{#1}}}
\newcommand{\Z}{\baseRing{Z}}
\newcommand{\R}{\baseRing{R}}
\newcommand{\C}{\baseRing{C}}

\newcommand{\pd}[2]{{\frac{\partial {#1}}{\partial {#2}}}}

\newcommand{\transpose}[1]{{#1}^t}
\newcommand{\jdef}[1]{\emph{#1}}

\newcommand{\idM}{\ensuremath{\baseRing{I}}}

\newcommand{\stext}[1]{\ensuremath{\quad \text{#1}\quad}}

\theoremstyle{plain}
\newtheorem{theorem}{Theorem}
\newtheorem{lemma}{Lemma}

\newtheorem{prop}{Proposition}

\theoremstyle{definition}
\newtheorem{defn}{Definition}

\newtheorem{example}{Example}

\newtheorem{remark}{Remark}

\DeclareMathOperator{\im}{Im}

\DeclareMathOperator{\gr}{Gr}

\DeclareMathOperator{\vspan}{Span}

\newcommand{\ie}{\textsl{i.e.}}
\newcommand{\del}{\ensuremath{\partial}}

\newcommand{\ti}[1]{\tilde{#1}}

\def\conj{\overline}

\def\GC{\ensuremath{G_\C}}
\def\GR{\ensuremath{G_\R}}
\def\jgk{\ensuremath{\mathfrak{k}}}
\def\jgp{\ensuremath{\mathfrak{p}}}
\def\gl{\ensuremath{\mathfrak{gl}}}
\def\jgsl{\ensuremath{\mathfrak{sl}}}
\def\lg{\ensuremath{\mathfrak{g}}}
\def\ga{\ensuremath{\mathfrak{a}}}

\def\lgc{\ensuremath{\mathfrak{g}_\C}}
\def\lgr{\ensuremath{\mathfrak{g}_\R}}

\def\DD{\ensuremath{\mathcal{D}}}
\def\DC{\ensuremath{\check{\mathcal{D}}}}

\def\IVI{IVI\xspace}
\def\IVIs{IVIs\xspace}

\CompileMatrices

\begin{document}

\title{Infinitesimal Variations of Hodge Structure\\ at Infinity}

\bibliographystyle{amsplain}

\author{Javier Fernandez} \address{Instituto Balseiro\\ Universidad
  Nacional de Cuyo -- C.N.E.A.\\ Bariloche \\R8402AGP \\Rep\'ublica Argentina}
\email{jfernand@ib.edu.ar}

\author{Eduardo Cattani}
\address{Department of Mathematics\\ University of Massachusetts\\ Amherst\\  
  MA 01003-9305\\USA}
\email{cattani@math.umass.edu}


\begin{abstract}
  By analyzing the local and infinitesimal behavior of degenerating
  polarized variations of Hodge structure the notion of infinitesimal
  variation of Hodge structure at infinity is introduced. It is shown
  that all such structures can be integrated to polarized variations
  of Hodge structure and that, conversely, all are limits of
  infinitesimal variations of Hodge structure (IVHS) at finite
  points. As an illustration of the rich information encoded in this
  new structure, some instances of the maximal dimension problem for
  this type of infinitesimal variation are presented and contrasted
  with the ``classical'' case of IVHS at finite points.
\end{abstract}


\maketitle

\section{Introduction}
\label{sec:introduction}

It is a well known fact in mathematics that most of the interesting
features of a map are encoded in its singular behavior. The Hodge
theoretic version of the previous statement is that the interesting
features of a polarized variation of Hodge structure (PVHS) can be
described by analyzing its degenerating behavior. The purpose of this
note is to start the exploration of the first order behavior of PVHS,
from the perspective of a degenerating point, that is, a point at
infinity.

J. Carlson, M. Green, P. Griffiths, and J. Harris introduced
in~\cite{ar:CGGH-IVHS} the idea of infinitesimal variation of Hodge
structure (IVHS) as a way of associating to a PVHS an object with
interesting linear algebraic invariants. Basically, if we represent
locally a PVHS as an integral manifold of Griffiths' exterior
differential system as in~\cite{ar:CGGH-IVHS}
or~\cite{ar:Mayer-coupled}, then an IVHS is an integral element of the
differential system.

Since IVHS are useful in the analysis of PVHS, we introduce a similar
notion associated to degenerating PVHS. Infinitesimal variations of
Hodge structure at infinity (\IVI) are introduced in
Definition~\ref{def:ivi}.  This idea is already implicit in the
construction of compactifications in~\cite{ar:CCK} as well as
in~\cite{ar:KU-compactification}.

Using the local description of PVHS near
infinity~\cite{ar:S-vhs,ar:CK-luminy,ar:CF-asymptotics} we are able to
prove in Theorem~\ref{th:var_to_abel} that every \IVI integrates to a
PVHS and that every \IVI is a limit of IVHS at infinity.

We claim that \IVIs encode more refined information of a PVHS than an
IVHS does for finite points. In order to illustrate this statement we
look into the maximal dimension problem for \IVIs, and contrast this
case with the results known for IVHS.

Given the weight and Hodge numbers of a PVHS, J. Carlson, A. Kasparian
and D. Toledo in~\cite{ar:CKT} and R. Mayer in~\cite{ar:Mayer-coupled}
find sharp upper bounds for the maximal dimension of an IVHS, which we
call the CKTM bounds. Even though the CKTM bounds remain valid and
sharp for \IVIs if one considers all possible nilpotent orbits on a
given period domain, considering only \IVIs with underlying particular
mixed Hodge structures or nilpotent cones leads to lower maximal
dimensions, corresponding to the stronger control that the nilpotent
orbit imposes on the possible \IVIs.

Finally, we see that, in some cases, for a given nilpotent orbit, there
are non-conjugate \IVIs of maximal dimension, a phenomenon that
doesn't occur for IVHS due to rigidity.

\smallskip

Section~\ref{sec:preliminaries} briefly reviews some results in
asymptotic Hodge theory and infinitesimal variations of Hodge
structure. Section~\ref{sec:IVIs} introduces the notion of
infinitesimal variation of Hodge structure at infinity and studies the
integrability of such objects. Section~\ref{sec:abelian_subalgebras}
explores some properties of the maximal dimension problem for \IVIs.


\section{Preliminaries}
\label{sec:preliminaries}

The study of the degenerating behavior of variations of Hodge
structure is the result of the work of P. Griffiths, P. Deligne,
W. Schmid, E. Cattani and A. Kaplan, among others. We refer
to~\cite{ar:CK-luminy} for a description of the subject as well as
references to the original papers.

We consider a finite dimensional $\R$-vector space $V_\R$ and its
complexification $V=\C \otimes V_\R$ with the induced conjugation
$v\mapsto \conj{v}$. A \jdef{(real) Hodge structure} (HS) of weight
$k$ on $V_\R$ is defined by a grading $H^{*,k-*}$ of $V$ subject to
the conditions
\begin{equation*}
  V = \oplus_a H^{a,k-a} \stext{ and } H^{a,k-a}= \conj{H^{k-a,a}}
  \stext{ for all } a.
\end{equation*}
The numbers $h^{a,k-a} = \dim H^{a,k-a}$ are called the \jdef{Hodge
  numbers} of the structure. The subspaces $F^a = \oplus_{b\geq a}
H^{b,k-b}$ form a decreasing filtration of $V$. Conversely, given such
a filtration subject to $V = F^a \oplus \conj{F^{k-a+1}}$ for all $a$,
the subspaces $H^{a,k-a} = F^a\cap \conj{F^{k-a}}$ define a HS of
weight $k$ on $V_\R$.

A \jdef{polarized Hodge structure} (PHS) of weight $k$ on $V_\R$ is
given by a HS of weight $k$ on $V_\R$, $H^{*,k-*}$, and a
nondegenerate bilinear form $Q$ on $V$ defined over $\R$, such that:
\begin{enumerate}
\item $Q(u,v)= (-1)^k Q(v,u)$ for all $u,v\in V$,
\item \label{it:orthogonality} $Q(H^{a,k-a},H^{b,k-b})= 0$ if $a+b\neq
  k$ and
\item \label{it:positivity} the Hermitian form $Q(C_{H} \cdot ,
  \conj{\cdot})$ is positive definite, where $C_{H} v = i^{a-b} v$ for
  $v \in H^{a,b}$.
\end{enumerate}

If $V$ is the vector space underlying a PHS with form $Q$, we denote
by $\GC=O(V,Q)$ the group of isometries and by $\GR$ the subgroup of
$\GC$ preserving $V_\R$. The Lie algebras of $\GC$ and $\GR$ are
denoted by $\lgc$ and $\lgr$ respectively. 

All PHS of weight $k$ and fixed Hodge numbers are parametrized by a
space denoted by $\DD$ (see~\cite[\S 3]{ar:S-vhs}). In order to
describe $\DD$ we consider the space of all flags of $V$ of
appropriate dimension that satisfy the orthogonality
condition~(\ref{it:orthogonality}). This is a subvariety of the
corresponding flag manifold; it is called the \jdef{compact dual
  space} of $\DD$ and is denoted by $\DC$. $\GC$ acts transitively on
$\DC$, so that $\DC\simeq \GC/B$ where $B\subset \GC$ is a parabolic
subgroup.  The space $\DD \subset\DC$ corresponds to all flags that
satisfy, in addition to~(\ref{it:orthogonality}), the positivity
condition~(\ref{it:positivity}). $\DD$ is an open subspace and $\GR$
acts transitively on $\DD$.

A filtration $F^*\in\DC$ defines a filtration of $\lgc$ by $F^a\lgc
=\{X\in \lgc: X(F^b)\subset F^{a+b}\}$. When $F^*\in\DD$, it also
defines a Hodge structure of weight $0$ on $\lgr$ with the grading of
$\lgc$ given by 
\begin{equation}
  \label{eq:grading_of_lgc}
  \lgc^{s,-s} = \{ X\in \lgc: X(H^{a,k-a})\subset H^{a+s,k-a-s}\},
\end{equation}
where $H^{*,k-*}$ is the grading associated to $F^*$.

A \jdef{mixed Hodge structure} (MHS) on $V_\R$ consists of a pair of
filtrations of $V$, $(W_*, F^*)$, $W_*$ real and increasing, $F^*$
decreasing, such that $F^*$ induces a HS of weight $a$ on
$\gr_a^{W_*}$ for each $a$.

Given any nilpotent $N\in\gl(V)$, there is a filtration $W(N)_*$ of $V$ called
its \jdef{weight filtration} (see~\cite[page 468]{ar:CKS}). This
filtration is the unique increasing filtration that satisfies
$N(W_l)\subset W_{l-2}$ and $N^l:\gr_{l}^{W_*}\rightarrow
\gr_{-l}^{W_*}$ is an isomorphism.

A \jdef{polarized mixed Hodge structure}
(PMHS)~\cite[1.16]{ar:CK-luminy} of weight $k$ on $V_\R$ consists of a
MHS $(W_*,F^*)$ on $V_\R$, a nilpotent element $N\in (F^{-1}\lgc \cap
\lgr)$ and a nondegenerate bilinear form $Q$ such that
\begin{enumerate}
\item $N^{k+1}=0$,
\item $W_* = (W(N)[-k])_*$, where $W[-k]_j = W_{j-k}$,
\item $Q(F^a,F^{k-a+1}) = 0$ and,
\item the HS of weight $k+l$ induced by $F^*$ on
  $\ker(N^{l+1}:\gr_{k+l}^{W_*}\rightarrow \gr_{k-l-2}^{W_*})$ is
  polarized by $Q(\cdot,N^l \cdot)$.
\end{enumerate}

A polarized variation of HS (PVHS)~\cite[Section 1]{ar:CK-luminy} over
a manifold $M$ determines a holomorphic map $\Phi:M\rightarrow
\DD/\Gamma$, where $\Gamma\subset \GC$ is a discrete subgroup; $\Phi$
is called the \jdef{period mapping}. The map $\Phi$ is also locally
liftable and horizontal.  Horizontality in this context means that
$\im d\Phi$ is contained in the $\GC$-homogeneous subbundle $T_h\DC
\subset T\DC$ with fiber, over $F^*\in\DC$, given by
$F^{-1}\lgc/F^0\lgc \subset \lgc/F^0\lgc$; $T_h\DC$ is called the
\jdef{horizontal bundle}. When $F^*\in\DD$, $(T_h\DC)_{F^*} =
\lgc^{-1,1}$ as defined by~\eqref{eq:grading_of_lgc}.

Our main interest is the asymptotic behavior of $\Phi$ near the
boundary of $M$, with respect to some compactification $\overline{M}$
where $\overline{M}-M$ is a divisor with normal crossings. Such
compactifications exist if, for instance, $M$ is
quasiprojective. Locally at infinity we may as well replace $M$ by
some product of punctured discs and discs, so that
\begin{equation*}
  \Phi:(\Delta^*)^r \times \Delta^m\rightarrow \DD/\Gamma.
\end{equation*}
Since the $r$th-power of the upper half plane, $U^r$ is the universal
cover of $(\Delta^*)^r$, we can lift $\Phi$ to $U^r\times \Delta^m$.
We still refer to this induced map by $\Phi$. We denote by $z=(z_j)$,
$t=(t_l)$ and $s=(s_j)$ the coordinates on $U^r$, $\Delta^m$ and
$(\Delta^*)^r$ respectively. By definition, we have $s_j = e^{2\pi i
  z_j}$.

A \jdef{nilpotent orbit} is a horizontal map
\begin{equation*}
  \theta:\C^r\rightarrow \DC,\, \theta(z) = \exp(\sum_{j=1}^r z_j
  N_j)\cdot F^*
\end{equation*}
where $F^*\in\DC$, $\{N_1,\ldots,N_r\} \subset (F^{-1}\lgc\cap \lgr)$ is a
commuting subset of nilpotent elements and there is $\alpha\in\R$ such
that $\theta(z)\in\DD$ for $\im(z_j) > \alpha$. We usually denote a
nilpotent orbit by $\{N_1,\ldots,N_r;F^*\}$ and the cone
$C(N_1,\ldots,N_r) = \{\sum \lambda_j N_j: \lambda_j \in\R_{>0}\}$ is
called the \jdef{nilpotent cone} of the orbit. Even if nilpotent
orbits are analytic objects by definition, they can be algebraically
characterized as asserted by the following result (\cite[Theorem
2.3]{ar:CK-luminy}) that, in turn, puts together important results of
several authors.

\begin{theorem} \label{th:2.3}
  If $\{N_1,\ldots,N_r;F^*\}$ is a nilpotent orbit, then
  \begin{enumerate}
  \item \label{it:2.3.1} $N_j^{k+1} = 0$ where $k$ is the weight of
    the PHS in $\DD$.
  \item \label{it:2.3.2} Every $N\in C(N_1,\ldots,N_r)$ defines the same weight
    filtration $W^C_*$.
  \item \label{it:2.3.3} $((W^C[-k])_*,F^*)$ is a PMHS, polarized by
    every $N\in C(N_1,\ldots,N_r)$.
  \end{enumerate}
  Conversely, if $F^* \in \DC$, and $\{N_1,\ldots,N_r\}$ are commuting
  nilpotent elements of $F^{-1}\lgc \cap \lgr$ that satisfy the
  conditions~\ref{it:2.3.1}, \ref{it:2.3.2} and~\ref{it:2.3.3} for
  some $N\in C(N_1,\ldots,N_r)$, then $\{N_1,\ldots,N_r;F^*\}$ is a
  nilpotent orbit.
\end{theorem}

By Schmid's Nilpotent Orbit Theorem~\cite[4.12]{ar:S-vhs}, there is
a nilpotent orbit $\{N_1,\ldots,N_r;F^*\}$ associated to any
degenerating PVHS $\Phi$; in this case, $N_j$ is the logarithm of the
unipotent part of the monodromy. 

In order to make the relationship between period mappings and their
nilpotent orbits more precise, recall that there is a canonical
bigrading $\{I^{*,*}\}$ associated with any MHS $(W_*,F^*)$. It is
uniquely characterized by the property $I^{p,q}\equiv \conj{I^{q,p}}
\mod(\oplus_{a<p,b<q} I^{a,b})$ (see~\cite[2.13]{ar:CKS}). This
bigrading induces, in turn, a bigrading $I^{*,*}\lgc$ of
$(W_*\lgc,F^*\lgc)$.

Set
\begin{equation}
  \label{eq:def_pa}
  \jgp_a =\oplus_{q}I^{a,q}\lgc \qquad \text{ and } \qquad \lg_- = \oplus_{a\leq
    -1}\jgp_a.
\end{equation}
It is immediate that if $(W_*,F^*)$ is a MHS, $(T_h\DC)_{F^*} =
F^{-1}\lgc/F^0\lgc \simeq \jgp_{-1}$. Also, $\lgc=\lg_- \oplus
Stab_{\GC} (F^*)$, so that $(\lg_-, X \mapsto \exp(X)\cdot F^*)$
provides a local model for the $\GC$-homogeneous space $\DC$ near
$F^*$. We recall from~\cite[Section 2]{ar:CK-luminy} that we can
represent a degenerating PVHS $\Phi$ by
\begin{equation}
  \label{eq:phi_g}
  \Phi(z,t) = \exp\bigg(\sum_{j=1}^r z_j N_j\bigg)\cdot \exp(\Gamma(\exp(2\pi i
  z),t)) \cdot F^*
\end{equation}
where $(N_1,\ldots,N_r;F^*)$ is the nilpotent orbit and
$\Gamma:\Delta^r\times \Delta^m\rightarrow \lg_-$ is holomorphic. It
is possible to rewrite 
\begin{equation}
  \label{eq:phi_x}
  \Phi(z,t) = \exp(X(z,t))\cdot F^*
\end{equation}
for a holomorphic $X:U^r\times \Delta^m\rightarrow \lg_-$. In particular, 
\begin{equation}
  \label{eq:X-1_gamma}
  X_{-1}(z,t) = \sum_{j=1}^r z_j N_j +
  \Gamma_{-1}(\exp(2\pi i z),t),  
\end{equation}
where the $-1$ subscript denotes the $\jgp_{-1}$ part of the
corresponding application.

In terms of $X$, recalling the $\GC$-homogeneity of $T_h\DC$, the
horizontality of $\Phi$ is expressed by:
\begin{equation*}
  \exp(-X)\, d\exp(X) \in \jgp_{-1}\otimes T^*(U^r\times \Delta^m).
\end{equation*}
In fact, since the $\jgp_{-1}$-part of $\exp(-X)\, d\exp(X)$ is
$dX_{-1}$, the horizontality condition can be written as 
\begin{equation}
  \label{eq:horiz}
  \exp(-X)\, d\exp(X) = dX_{-1}.
\end{equation}
It follows from this last expression that
\begin{equation}
  \label{eq:integcond}
  dX_{-1}\wedge dX_{-1}\ =\ 0.
\end{equation}

The following result, which follows from
\cite[Theorem~2.8]{ar:CK-luminy} and
\cite[Theorem~2.7]{ar:CF-asymptotics}, shows that the nilpotent orbit
together with the $\jgp_{-1}$-valued holomorphic function
$\Gamma_{-1}$ completely determine the local behavior of the
variation:

\begin{theorem}\label{th:improved_2.8}
  Let $\{N_1,\ldots,N_r;F_0\}$ be a nilpotent orbit and $R:\Delta^r
  \times \Delta^m \rightarrow \jgp_{-1}$ be a holomorphic map with
  $R(0,0)=0$.  Define $X_{-1}(z,t) = \sum_{j=1}^r z_j N_j+R(s,t)$,
  $s_j = e^{2\pi i z_j}$, and suppose that the differential equation
  \eqref{eq:integcond} holds.  Then, there exists a unique period
  mapping
  \begin{equation*}
    \Phi(s,t) \ =\ \exp\biggl(\frac{1}{2\pi i} \sum_{j=1}^r \log(s_j)
    N_j\biggr)\cdot \exp(\Gamma(s,t))\cdot F_0,
  \end{equation*}
  defined in a neighborhood of the origin in $\Delta^{r+m}$ such that
  $\Gamma_{-1} = R$.
\end{theorem}

The importance of this last theorem is that the information contained
in a degenerating PVHS is encoded in the data of a nilpotent orbit and
a holomorphic map satisfying the integrability
condition~\eqref{eq:integcond}.


Last, we turn to the first order content of a PVHS. The analysis of
the differential of the period mapping at a point $F^*\in\DD$ led to
the following definition~\cite[\S 1.c]{ar:CGGH-IVHS}. An
\jdef{infinitesimal variation of Hodge structure} (IVHS) at
$F^*\in\DD$ consists of a pair $(T,\delta)$, where $T$ is a finite
dimensional vector space and $\delta\in\hom(T,(T_h\DC)_{F^*}) =
\hom(T,\lgc^{-1,1})$ such that $\im(\delta)$ is an abelian subspace of
$\lgc^{-1,1} \subset\lgc$. In other words:
\begin{gather}
  \delta:T \rightarrow \oplus_p \hom(H^{p,k-p},H^{p-1,k-p+1})
  \stext{ is linear}\\
  \delta(\xi_1)\circ \delta(\xi_2) = \delta(\xi_2)\circ \delta(\xi_1)
  \stext{for all} \xi_1,\xi_2\in T\label{eq:ivhs_abelian}\\
  Q(\delta(\xi)\psi,\eta)+Q(\psi,\delta(\xi)\eta)= 0 \stext{for all}
  \psi\in H^{p,k-p},\, \eta\in H^{p-1,k-p+1}.
\end{gather}
The $1$-forms that annihilate $T_h\DC$ generate the differential ideal
that is known as \jdef{Griffiths' exterior differential system}. It
turns out that, because of~\eqref{eq:ivhs_abelian}, all IVHS are
integral elements of that system. It is also known~\cite[Proposition
3.15]{ar:Mayer-coupled} that every integral element of Griffiths'
system can be integrated to a germ of an integral manifold of the
system which, in Hodge theoretic terms, says that all IVHS arise from
(germs of) PVHS.


\section{Infinitesimal variations at infinity}
\label{sec:IVIs}

By analogy with the ``classical'' study of the first order behavior of
a PVHS at a point $F^*\in\DD$ via IVHS, we want to analyze the first
order behavior of a degenerating PVHS at a point $F^*\in\DC$, that is,
at infinity.

Suppose that $\Phi$ is a degenerating PVHS with nilpotent orbit
$\{N_1,\ldots, N_r;F^*\}$. We will study how the tangent spaces to the
image of $\Phi$ ---that is, the IVHSs associated to $\Phi$--- behave
as $\Phi$ degenerates at $F^*$.

We find the tangent spaces to the variation by computing $d\Phi_{w_0}$
for $w_0$ in a neighborhood $W$ of infinity. Following the description
given in~\cite[pages 17 and 18]{bo:griffiths-topics} we consider:
\begin{equation*}
  d\Phi_{w_0}:(TW)_{w_0}\rightarrow (T_h\DC)_{\Phi(w_0)} \subset
  \oplus_a \hom(\gr^{\Phi(w_0)}_a,\gr^{\Phi(w_0)}_{a-1}).
\end{equation*}
If $\{I^{*,*}\}$ is the bigrading associated to the limiting MHS of
$\Phi$ at $F^*\in\DC$, the subspaces $J^*=\oplus_q I^{*,q}\subset V$
form a grading of $F^*$.  Using the $\GC$-action on $V$ and the
form~\eqref{eq:phi_x} for $\Phi$, we define $L^* = \exp(X(w_0))J^*$, a
grading of $\Phi(w_0)$. There are then isomorphisms
$\hom(\gr^{\Phi(w_0)}_a,\gr^{\Phi(w_0)}_{a-1}) \simeq
\hom(L^a,L^{a-1}) \simeq \hom(J^a,J^{a-1})$, with the last isomorphism
being conjugation by $\exp(X(w_0))$. Putting together the different
identifications we have
\begin{equation*}
  d\Phi_{w_0}:(TW)_{w_0}\rightarrow \oplus_a \hom(J^a,J^{a-1})
  \subset \lgc.
\end{equation*}
  
We claim that under the last representation, $d\Phi = dX_{-1}$. To
prove the claim, we have to show that $d\Phi_{w_0}= \exp(X(w_0))
dX_{-1}|_{w_0} \exp(-X(w_0))$.

Observe that if $\pi_{J^a}$ denotes the projection from $V$ onto the
$J^a$ factor and, analogously, for the $L^*$ grading, $\pi_{L^a} =
e^{X(w_0)} \pi_{J^a} e^{-X(w_0)}$.

Define $\del_j = \pd{}{w_j}|_{w_0}$. In order to compute $d\Phi_{w_0}
(\del_j) (\exp(X(w_0)) v_0)$ for $v_0\in J^a$, following the
computation described on pages 17 and 18
of~\cite{bo:griffiths-topics}, a curve through $\exp(X(w_0)) v_0$ and
contained in $\exp(X(w))J^a$ for $w$ near $w_0$ is needed. Such a curve can
be constructed considering $\exp(X(w)) v_0$. All together:
\begin{gather*}
  \exp(-X(w_0)) d\Phi_{w_0} (\del_j) \exp(X(w_0)) v_0 =
  \exp(-X(w_0)) \pi_{L^{a-1}}(\del_j (\exp(X(w)) v_0)) \\
  =\pi_{J^{a-1}}(\exp(-X(w_0)) (\del_j (\exp(X(w)) v_0))) \\ =
  \pi_{J^{a-1}}(\exp(-X(w_0))(\del_j(\exp(X(w))) v_0)).
\end{gather*}
By~\eqref{eq:horiz}, we have $\exp(-X(w_0)) \del_j(\exp(X(w))) =
\del_j(X_{-1}) \in \jgp_{-1}$.  So that
\begin{equation*}
  \exp(-X(w_0)) d\Phi_{w_0} (\del_j) \exp(X(w_0)) v_0 = dX_{-1}|_{w_0}
  v_0
\end{equation*}
and the claim follows. Therefore, using~\eqref{eq:X-1_gamma},
\begin{equation*}
  \im d\Phi_{(s,t)} = \im dX_{-1}|_{(s,t)} =
  \vspan_\C\bigg(N_j+\pd{\Gamma_{-1}}{s_j}\bigg|_{(s,t)} 2\pi i s_j,
  \pd{\Gamma_{-1}}{t_l}\bigg|_{(s,t)} \text{ for all } j,l\bigg).
\end{equation*}

Since $s=0$ is an accumulation point of points where $d\Phi$ has
maximal rank, we can consider the limit of the corresponding tangent
spaces which, by the holomorphicity of $\Gamma_{-1}$ at $(0,0)$,
satisfy
\begin{equation*}
  \lim_{s\rightarrow 0, t\rightarrow 0}  \im d\Phi_{(s,t)} = 
  \vspan_\C\bigg(N_j, \pd{\Gamma_{-1}}{t_l}\bigg|_{s=0,t=0}
  \text{ for all } j,l\bigg).
\end{equation*}

If we let
\begin{equation}
  \label{eq:def_ga}
  \ga_\infty^\Phi = \vspan_\C\bigg(N_j, \pd{\Gamma_{-1}}{t_l}\bigg|_{s=0,t=0}
  \text{ for all } j,l\bigg),
\end{equation}
we have just seen that the IVHSs associated to $\Phi$ ---their images
in the corresponding Grassmanian--- have $\ga_\infty^\Phi$ as a limit
point. Since all those subspaces are abelian, we conclude that
$\ga_\infty^\Phi$ is also abelian.

Abstracting the features of $\ga_\infty^\Phi$ we arrive to the
following definition.
\begin{defn}
  \label{def:ivi}
  Given a period domain $\DD$, an \jdef{infinitesimal variation of
    Hodge structure at infinity\/} (\IVI) is a pair $(\{N_1,\ldots,N_r
  ; F^*\},\ga)$, where $\{N_1,\ldots,N_r ; F^*\}$ is a nilpotent orbit
  in $\DC$ and, for $\jgp_{-1}$ is defined by~\eqref{eq:def_pa},
  $\ga\subset \jgp_{-1}$ is an abelian subspace such that
  $\vspan_\C{(N_1,\ldots,N_r)}\subset \ga$. The dimension of
  an \IVI is the dimension of $\ga$.
\end{defn}

Our previous discussion can be extended now to the following result
relating \IVIs to degenerating PVHS.

\begin{theorem}\label{th:var_to_abel}
  Let $\Phi:U^r \times \Delta^m\rightarrow \DD$ be a degenerating PVHS
  with nilpotent orbit $\{N_1,\ldots,N_r ; F^*\}$. Also, let
  $\ga_\infty^\Phi$ be defined by~\eqref{eq:def_ga}, where
  $\Gamma_{-1}$ is the holomorphic function associated to $\Phi$
  by~\eqref{eq:phi_g}.  Then, $(\{N_1,\ldots,N_r ;
  F^*\},\ga_\infty^\Phi)$ is an \IVI. Moreover, every \IVI is of this
  form.
\end{theorem}

\begin{proof}
  The previous discussion shows that $\ga_\infty^\Phi$ is an abelian
  subspace of $\lgc$ that, by construction, is contained in $\jgp_{-1}$
  and contains the nilpotent cone of the nilpotent orbit
  $\{N_1,\ldots, N_r;F^*\}$ associated with $\Phi$. Therefore,
  $(\{N_1,\ldots,N_r ; F^*\},\ga_\infty^\Phi)$ is an \IVI.

  Conversely, if $(\{N_1,\ldots,N_r;F^\cdot\},\ga)$ is an \IVI, let
  $\{B_1,\ldots, B_l\}$ be a basis of a complement of
  $\vspan_\C(N_1,\ldots,N_r)$ in $\ga$. Then define the map
  $\Gamma_{-1}:\Delta^r\times\C^l\rightarrow \jgp_{-1}$ by
  $\Gamma_{-1}(s,t)=\sum_{j=1}^l t_j B_j$.
  
  Define $X_{-1}(z,t)=\sum_{j=1}^r z_j N_j +\Gamma_{-1}(s,t)$, as
  in~\eqref{eq:X-1_gamma}. Then, since all the elements of $\ga$
  commute with each other, condition~\eqref{eq:integcond} holds so
  that, by Theorem~\ref{th:improved_2.8}, $X_{-1}$ extends to a
  degenerating PVHS $\Phi$ with the given nilpotent orbit.  Since
  $\pd{\Gamma_{-1}}{t_l} = B_l$, it follows that the given \IVI arises
  from $\Phi$.
\end{proof}


\begin{remark}
  As part of the previous analysis we found that, under the
  corresponding identifications, $ d\Phi = d X_{-1}$. In fact, under
  this interpretation, the integrability
  condition~\eqref{eq:integcond} says that $d X_{-1}$, and then
  $d\Phi$, are Higgs fields.
\end{remark}

The notion of \IVI introduced above is richer than that of IVHS in
that it encodes information about the nilpotent orbit as well as the
holomorphic part of a degenerating PVHS. In the next section we will
illustrate this statement with several examples.


\section{Abelian subalgebras}
\label{sec:abelian_subalgebras}

IVHS have appeared in connection with several geometric problems
including Torelli theorems, the Noether-Lefshchetz theorems and the
Yukawa coupling (in physics!). Another application has been the study
of bounds on the dimension of variations of Hodge structure, as
started by J. Carlson in~\cite{ar:carlson-bounds_2}. In this section
we want to illustrate the notion of \IVI by contrasting some examples
and results with those available for IVHS.

The problem of classifying IVHS is quite hard. Still, the following
result holds (\cite[Theorem 1.6]{ar:CKT}, \cite[Theorem
4.15]{ar:Mayer-coupled}).
\begin{theorem}\label{th:CKTM}
  If $F^*\in\DD$ and $\ga\subset \lgc^{-1,1}\subset \lgc$ is an
  abelian subspace, then $\dim \ga \leq q(k,h)$, where $q$ is an
  explicit piecewise quadratic function of the weight $k$ and the
  Hodge numbers $h^{p,q}$. Furthermore, the bounds are sharp.
\end{theorem}

In the same setting, R. Mayer generalized partial results of
Carlson~\cite{ar:carlson-bounds_2} to the effect that, except for some
small dimensional cases, maximal dimensional abelian subalgebras
generate rigid variations~\cite[Theorem 5.1]{ar:Mayer-coupled}.

Our first observation is that \IVIs satisfy the bounds of
Theorem~\ref{th:CKTM}. Indeed, by Theorem~\ref{th:var_to_abel}, any
\IVI can be integrated to a PVHS of the same dimension. Moreover, we
will show below that there are \IVIs of the maximal dimension.

\begin{prop}\label{prop:CKTM_bound_via_IVI}
  For weight $k=2$ and any Hodge numbers $h^{*,k-*}$, there are \IVIs
  whose dimension is $q(k,h)$.
\end{prop}

\begin{remark}
  For simplicity, we are stating
  Proposition~\ref{prop:CKTM_bound_via_IVI} only for $k=2$. The result
  can be proved for arbitrary weight using the same techniques, as it
  is described in Remark~\ref{re:CKTM_bound_via_IVI}.
\end{remark}

Before we start the proof of Proposition~\ref{prop:CKTM_bound_via_IVI}
we will make explicit the bounds of
Theorem~\ref{th:CKTM}~\cite[Theorem 1.6]{ar:CKT}, in the case $k=2$.
\begin{enumerate}
\item $h^{2,0}>1: q(2,h)=
  \begin{cases}
    \frac{1}{2} h^{2,0} (h^{1,1}-1) + 1 \text{, if $h^{1,1}$ is odd},\\
    \frac{1}{2} h^{2,0} h^{1,1} \text{, if $h^{1,1}$ is even}.
  \end{cases}$
\item $h^{2,0}=1: q(2,h)=h^{1,1}$.
\end{enumerate}

The following technical results are needed to prove
Proposition~\ref{prop:CKTM_bound_via_IVI}.

\begin{lemma}\label{le:orthogonal_implies_rest}
  Let $V$ be a vector space underlying a PHS of weight $k$ with
  polarizing form $Q$ and Hodge numbers $h^{*,k-*}$. Suppose that
  $\{J^{*,*}\}$ is a bigrading of $V$ and $j^{a,b} = \dim J^{a,b}$ so
  that the following properties are satisfied:
  \begin{enumerate}
  \item \label{it:construction_1} $J^{a,b} = \conj{J^{b,a}}$ for all
    $a,b$.
  \item \label{it:construction_2} $j^{a,b} = j^{k-b,k-a}$ for all
    $a,b$.
  \item \label{it:construction_3} $h^{p,k-p} = \sum_{b} j^{p,b}$ for
    all $p$.
  \item \label{it:construction_4} $j^{a+1,b+1}\leq j^{a,b}$ for all
    $a,b$ with $a+b\geq k$.
  \item \label{it:construction_5} $Q(J^{a,b},J^{a',b'})=0$ unless
    $a+a'=k$ and $b+b'=k$.
  \end{enumerate}
  Then, if $F^p = \oplus_{a\geq p ,b} J^{a,b}$, and $W_l =
  \oplus_{a+b\leq l} J^{a,b}$, there exist $N\in \lgr$ such that
  $(W[-k]_*,F^*,N)$ is a PMHS.
\end{lemma}

\begin{proof}
  This is only a sketch: the details are an exercise in linear
  algebra. First notice that, by condition~\ref{it:construction_1},
  $(W_*,F^*)$ defines a MHS split over $\R$. Then,
  conditions~\ref{it:construction_1}, \ref{it:construction_2},
  \ref{it:construction_3} and \ref{it:construction_4} imply the
  existence of $N\in\gl(V_\R)$ such that $W_* = W(N)_*$ and $N$ is a
  $(-1,-1)$ morphism of the MHS. Last,
  condition~\ref{it:construction_5} implies that $N$ can be chosen in
  $\lgr$ and so that $(W[-k]_*,F^*,N)$ is a PMHS.
\end{proof}

\begin{lemma}\label{eq:dimensions_implies_rest}
  For any weight, $k$, and Hodge numbers, $h^{*,k-*}$, let
  $j^{a,b}\in\Z_{\geq 0}$ be such that
  conditions~\ref{it:construction_2}, \ref{it:construction_3} and
  \ref{it:construction_4} in Lemma~\ref{le:orthogonal_implies_rest}
  hold. Then there are bigradings $\{J^{*,*}\}$ of $V$ with $j^{a,b} =
  \dim J^{a,b}$ such that the rest of the conditions of
  Lemma~\ref{le:orthogonal_implies_rest} hold.
\end{lemma}

\begin{remark}
  Combining Lemmas~\ref{le:orthogonal_implies_rest}
  and~\ref{eq:dimensions_implies_rest} we see that it suffices to set
  dimensions satisfying adequate compatibility conditions to ensure
  the existence of PMHS with bigrading of dimensions given by the
  given data.
  
  Notice that by the symmetry conditions, it is sufficient to set
  compatible values of $j^{a,b}$ for $a+b\geq k$ and $a\geq b$.  In
  what follows, we will usually set the values of some $j^{a,b}$, with
  the others determined either by symmetry or, otherwise, are $0$.
\end{remark}

\begin{proof}[Proof of Proposition~\ref{prop:CKTM_bound_via_IVI}]
  
  We start with the case $h^{2,0}>1$.

  Suppose $h^{1,1}$ is odd. There are two possibilities to consider:
  \begin{itemize}
  \item $2 h^{2,0}>h^{1,1}-1$. By
    Lemma~\ref{eq:dimensions_implies_rest}, given dimensions $j^{2,1}
    = \frac{1}{2}(h^{1,1}-1)$, $j^{2,0} = h^{2,0} -
    \frac{1}{2}(h^{1,1}-1)$, $j^{1,1} = 1$ there are MHS $\{J^{*,*}\}$
    with the right dimensions, polarized by some $N$. As an
    illustration of the ideas used in the proof of
    Lemma~\ref{eq:dimensions_implies_rest}, we will construct the
    bigrading $\{J^{*,*}\}$ explicitly. $Q$ has signature
    $(2h^{2,0},h^{1,1})$ so that we can split $V=V_1\oplus V_2 \oplus
    V_3$ with $\dim V_1 = 1$, $\dim V_2 = 2(h^{1,1}-1)$ and $\dim V_3
    = 2 h^{2,0}-h^{1,1}+1$ and so that the signature of $Q|_{V_i}$ is
    $(0,1)$, $(h^{1,1}-1, h^{1,1}-1)$ and $(2 h^{2,0}-h^{1,1}+1,0)$
    respectively. Split $V_2 = I_1 \oplus I_2$ with $I_j$ real and
    isotropic, $\dim I_j= h^{1,1}-1$. Notice that $Q$ induces $I_2
    \simeq I_1^*$. Write $I_1 = K_1 \oplus \conj{K_1}$, and under the
    previous isomorphism $I_2 = K_2 \oplus \conj{K_2}$ where
    $K_2\simeq K_1^*$. Finally, $V_3 = W \oplus \conj{W}$ with $W$
    isotropic, $\dim W = h^{2,0}-\frac{1}{2}(h^{1,1}-1)$. Now define
    $J^{1,1}=V_1$, $J^{2,1} = K_1$, $J^{1,2} = \conj{K_1}$,
    $J^{1,0}=\conj{K_2}$, $J^{0,1} = K_2$, $J^{2,0} = W$ and $J^{0,2}=
    \conj{W}$. Then the bigrading $\{J^{*,*}\}$ satisfies the
    conditions of Lemma~\ref{le:orthogonal_implies_rest} so that it
    induces a PMHS.
    
    Any map $\ti{X}\in\hom(J^{2,1},J^{1,0})$ can be extended to a map
    $X\in\jgp_{-1}$ ($\jgp_{*}$ as defined by~\eqref{eq:def_pa}) such
    that $X|_{J^{1,2}}\in \hom(J^{1,2}, J^{0,1})$ is dual to $\ti{X}$,
    using condition~\ref{it:construction_5} in
    Lemma~\ref{le:orthogonal_implies_rest}, and $X$ vanishes elsewhere.
    The same argument shows that any $\ti{\phi}\in
    \hom(J^{2,0},J^{1,1})$ and $\ti{\psi}\in \hom(J^{2,0},J^{1,0})$
    extend to maps $\phi,\psi\in \jgp_{-1}$. All of this may be
    schematized as follows:
    \begin{equation}\label{eq:max_type}
      \xymatrix{& J^{2,1}\ar[dd]^(.3){X} & &
        J^{1,2}\ar[dd]^(.3){X}\ar[dr]^{\psi} & \\ J^{2,0}
        \ar[rr]^(.3){\phi} \ar[dr]^{\psi} & &
        J^{1,1}\ar[rr]^(.3){\phi} & & J^{0,2} \\ & J^{1,0} & & J^{0,1}
        &}
    \end{equation}
    Let $\ga_1$ and $\ga_2$ be respectively the spaces of all the maps
    $X$ and $\psi$ constructed as above. Clearly $\ga_1$ and $\ga_2$
    are abelian. Moreover, any map $\phi$ commutes with $\ga_1 \oplus
    \ga_2$. For a fixed $\phi\neq 0$, define the abelian subspace $\ga
    = \ga_1 \oplus \ga_2 \oplus \C \{\phi\} \subset \jgp_{-1}$. We have
    $\dim \ga = \dim \ga_1 +\dim \ga_2 + 1 = (j^{2,1})^2 + j^{2,0}
    j^{2,1} + 1 = h^{2,0}(\frac{h^{1,1}-1}{2})+1$.
    
    Any $N\in J^{-1,-1}\lgc$ automatically satisfies $N\in \ga_1$.
    For a given $N$ which polarizes the MHS under consideration, by
    Theorem~\ref{th:2.3}, there are nilpotent orbits
    $\{N_1,\ldots,N_r;F^*\}$ such that $N$ is in the relative interior
    of $C(N_1,\ldots,N_r)$. Since all $N_j\in J^{-1,-1}\lgc$,
    $N_j\in\ga_1$. Thus $(\{N_1,\ldots,N_r;F^*\},\ga)$ is an \IVI.

    Notice that in this case we can have a nilpotent cone of maximal
    dimension $(j^{2,1})^2$.
    
  \item $2 h^{2,0}\leq h^{1,1}-1$. Consider a PMHS whose bigrading
    satisfies $j^{2,1}= h^{2,0}$, $j^{1,1}=h^{1,1}-2h^{2,0}$.  As in
    the previous case, any $\ti{X}\in\hom(J^{2,1},J^{1,0})$ and
    $\ti{\phi}\in \hom(J^{2,1},J^{1,1})$ extend to maps $X,\phi \in
    \jgp_{-1}$. This is described by:
    \begin{equation*}
      \xymatrix{J^{2,1} \ar[dd]^{X}\ar[rd]^{\phi} & &
        J^{1,2}\ar[dd]^{X} \\ & J^{1,1} \ar[rd]^{\phi} & \\ J^{1,0} &
        & J^{0,1}}
    \end{equation*}
    Since $Q|_{J^{1,1}}$ is positive definite, we can choose a
    splitting $J^{1,1}=L\oplus K \oplus \conj{K}$ where $\dim L=1$,
    $K$ is isotropic and $L$ and $K$ are orthogonal.
    
    Let $\ga_1$ and $\ga_2$ be respectively the spaces of maps $X$ and
    $\tau$ generated by $\ti{X}\in\hom(J^{2,1},J^{1,0})$ and
    $\ti{\tau}\in \hom(J^{2,1},K)$. Then $\ga_1\oplus\ga_2$ is
    abelian. Also, any fixed map $\ti{\psi}\in\hom(J^{2,1},L)$ induces
    a map $\psi\in \jgp_{-1}$ that commutes with $\ga_1\oplus\ga_2$.
    Thus $\ga= \ga_1\oplus\ga_2 \oplus \C\{\psi\}\subset \jgp_{-1}$ is
    an abelian subspace with $\dim \ga = \dim \ga_1 + \dim \ga_2 +1 =
    (j^{2,1})^2 + \frac{1}{2} j^{2,1}(j^{1,1}-1) +1 =
    \frac{1}{2} h^{2,0}(h^{1,1} - 1) + 1$.

    The construction of the \IVI then concludes as in the previous
    case. 
  \end{itemize}

  The case with $h^{1,1}$ even is done along the same lines as
  $h^{1,1}$ odd.
  
  Now suppose that $h^{2,0}=1$. If $h^{1,1}\geq 2$, we consider a PMHS
  whose bigrading has $j^{2,1}=1$, $j^{1,1}=h^{1,1}-2$ and is
  polarized by $N$:
  \begin{equation*}
    \xymatrix{ J^{2,1}\ar[dd]^{N}\ar[rr]^\psi \ar[dr]^\phi & &
      J^{1,2}\ar[dd]^{N} \\ & J^{1,1}\ar[dr]^\phi & \\ 
      J^{1,0}\ar[rr]^\psi & & J^{0,1} }
  \end{equation*}
  Then, there is a space of dimension $h^{1,1} -2$ of maps $\phi$,
  together with $\C\{N\}$ and $\C\{\psi\}$ for any fixed $\psi\neq 0$.
  They all commute, making an abelian space of dimension $h^{1,1}$, as
  required.

  The cases when $h^{2,0}=1$ and $h^{1,1}<2$ are immediate but notice
  that there is no logarithmic part (\ie, $r=0$).
\end{proof}

\begin{remark}
  In a few places during in the proof of
  Proposition~\ref{prop:CKTM_bound_via_IVI} we picked a map among many
  possible choices. For instance, in the case shown in
  diagram~\eqref{eq:max_type}, we chose a map $\phi$ to enlarge the
  abelian subspace $\ga_1\oplus \ga_2$. One may ask if it could be
  possible to enlarge the resulting abelian subspace even more by
  adding other such maps. On general grounds, the answer is no, since
  the dimension of the resulting abelian subspace would have to
  satisfy the CKTM bounds, and the examples constructed in the proof
  are of maximal dimension. More explicitly, in the case of
  diagram~\eqref{eq:max_type}, if $\phi'$ is like $\phi$ and commutes
  with $\phi$, it is easy to check that $\phi'\in \C\{ \phi\}$, so
  that further enlargement is not possible.
\end{remark}

\begin{remark}\label{re:CKTM_bound_via_IVI}
  The proof of Proposition~\ref{prop:CKTM_bound_via_IVI} in arbitrary
  weight follows along similar lines. Indeed, the proof of the
  sharpness of the bound in~\cite{ar:CKT} is made by showing, for each
  set of Hodge numbers, a specific IVHS of the maximal dimension. This
  IVHS is constructed out of four types of basic examples, combined
  appropriately by direct sums and tensor products. Hence, it suffices
  to show that each one of these basic types can be realized by a
  \IVI. In the proof of Proposition~\ref{prop:CKTM_bound_via_IVI}, we
  introduced two of the four basic types needed. The other two are
  constructed similarly and analogous results for direct sums and
  tensor products complete the argument.
\end{remark}

The maximal dimension of IVHS depends, by Theorem~\ref{th:CKTM}, on
the weight $k$ and the Hodge numbers of the structures.
Proposition~\ref{prop:CKTM_bound_via_IVI} shows that by taking
appropriate nilpotent orbits and abelian spaces, it is always possible
to achieve the maximal dimension given by the CKTM bounds on a given
period domain $\DD$ with \IVIs. The next example will show that, for a
given period domain, the maximal dimension problem for \IVIs with a
given underlying MHS can be more stringent.

\begin{example}\label{ex:all_ivi}
  Here we describe all the classes of \IVIs arising as degenerations
  of PVHS of weight $k=2$ and $h^{2,0} = h^{1,1}=3$. 

  Table~\ref{tab:all_examples} shows the different possible cones and
  the maximal dimension of the \IVIs in each case.
 
  \begin{table}[htbp]
    \begin{center}
      \leavevmode
      \begin{tabular}{|c|c|c|}\hline
        $j^{*,*}$ & nilpotent cones & max dim of \IVI \\\hline
        $j^{2,0}=j^{1,1}=3$ & $\{0\}$ & $4$ \\\hline 
        $j^{2,1}=j^{1,1}=1$ and $j^{2,0}=2$ & 1 cone of dimension 1 &
        $4$ 
        \\\hline 
        $j^{2,2}=1$, $j^{2,0}=2$ and $j^{1,1}=3$ & cones
        of dimension 1, 2 and 3 & in all cases $3$ \\\hline 
        $j^{2,2}=j^{1,1}=j^{2,1}=j^{2,0}=1$ & cones of
        dimension 1 and 2 & in all cases $3$ \\\hline 
        $j^{2,2}=2$, $j^{2,0}=1$ and $j^{1,1}=3$ & cones of
        dimension 1, 2 and 3 & in all cases $3$ \\\hline 
        $j^{2,2}=j^{1,1}=3$ & cones of
        dimension 1, 2 and 3 & in all cases $3$ \\\hline
      \end{tabular}
      \caption{MHS, nilpotent cones and \IVIs obtained when $k=2$ and
        $h^{2,0}=h^{1,1}=3$}
      \label{tab:all_examples}
    \end{center}
  \end{table}
\end{example}

\begin{remark}
  The first two rows of Table~\ref{tab:all_examples} correspond to MHS
  where the maximal dimension of \IVIs coincides with the CKTM
  bound. The remaining rows correspond to MHS where the maximal
  dimension of \IVIs is smaller, for all nilpotent cones.
\end{remark}

It is conceivable that the notion of \IVI could be attached to that of
MHS rather than to the nilpotent orbit as we do. The following example
will show that for a given MHS, the maximal dimension of the \IVIs
still depends on the full nilpotent orbit.
\begin{example}
  Consider the MHS in weight $k=2$, defined by $j^{2,2} = j^{1,1} =
  2d$.  As in the previous constructions, any $\ti{\phi} \in
  \hom(J^{2,2},J^{1,1})$ induces a map $\phi\in\jgp_{-1}$.  The
  condition for the commutativity of any two such morphisms becomes
  $\transpose{\ti{\phi}}_2 \ti{\phi}_1 = \transpose{\ti{\phi}}_1
  \ti{\phi}_2$, where $\transpose{\ti{\phi}}$ is the dual map of
  $\ti{\phi}$ under $Q$. Fix real bases of $V$ where the bilinear form
  $Q$ is given by $\left(
    \begin{matrix}
      & & \idM_{2d}\\ & -\idM_{2d}& \\ \idM_{2d} & &\\
    \end{matrix}\right)$. With respect to such bases, the matrix of
  $\transpose{\ti{\phi}}$ is the transpose of the matrix of
  $\ti{\phi}$.  For $1\leq a \leq 2d$, define $\ti{N}_a \in
  \hom(J^{2,2},J^{1,1})$ with respect to the same bases as above, by
  the matrices $\ti{N}_a = E_{a,a}$ whose only nonzero entry is a 1 in
  the $(a,a)$-position. Define also $N_0= \sum_{a=1}^{2d} N_a$.
  Clearly $N_0$ polarizes the MHS $\{J^{*,*}\}$. Then,
  $\{N_0;J^{*,*}\}$ and $\{N_1,\ldots,N_{2d};J^{*,*}\}$ are nilpotent
  orbits whose associated MHS is $\{J^{*,*}\}$.
  
  Now, we want to find the maximal dimension of the \IVIs for these
  nilpotent orbits. 
  \begin{enumerate}

  \item $\{N_1,\ldots,N_{2d};J^{*,*}\}$. In this case, the commutation with
    all the $N_a$ forces the elements of the abelian subspace
    containing the cone to be given by diagonal matrices in
    $\C^{2d\times 2d}$. So, any maximal abelian subspace has, at most,
    dimension $2d$. Therefore, the maximal dimension of \IVIs with the
    given nilpotent orbit is $2d$.

  \item $\{N_0;J^{*,*}\}$. In this case, the abelian subspaces
    containing the nilpotent cone are simply those containing the
    identity matrix. This condition forces the matrices representing
    $\ti{\phi}$ to be symmetric. In particular, the subspace of all
    the matrices of the form
    \begin{equation*}
      \left(
      \begin{matrix}
        a\idM+ iB & B \\ B & a\idM-iB
      \end{matrix}
      \right)
    \end{equation*}
    for $B \in \C^{d\times d}$ symmetric and $a\in\C$, is abelian,
    contains the identity matrix and has dimension
    $\frac{1}{2}d(d+1)+1$. So, the maximal dimension of \IVIs having
    the given nilpotent orbit is, at least, $\frac{1}{2}d(d+1)+1$. It
    will follow from Proposition~\ref{prop:max_dim_symmetric_matrices}
    that this is, in fact, the maximal dimension.
  \end{enumerate}
  Finally, since for $d\geq 3$, $\frac{1}{2}d(d+1)+1 > 2d$, we
  conclude that the maximal dimension depends on the nilpotent orbit
  and not just on the MHS.
\end{example}

\begin{remark}
  In order to study the maximal dimension of \IVIs whose underlying
  MHS is fixed, it is enough to consider one dimensional variations,
  that is, nilpotent orbits with nilpotent cone generated by one
  element. This is so since if $(\{N_1,\ldots,N_r;F^*\},\ga)$ is an
  \IVI and $N=\sum a_j N_j$ then $\{N;F^*\}$ is a nilpotent orbit
  with the same underlying MHS, so that $(\{N;F^*\},\ga)$ is an \IVI
  of the same dimension with a one dimensional nilpotent cone.
\end{remark}

The general problem of finding the maximal dimension of \IVIs with a
given nilpotent orbit is quite complex. Below, we concentrate on one
particular case to see some of its features.

Consider the following MHS of weight $k$ of Hodge-Tate type polarized
by $N_0$:
\begin{equation}\label{eq:PMHS-HT}
  \xymatrix{{J^{k,k}} \ar[d]^{N_0} \\ {J^{k-1,k-1}}\ar[d]^{N_0} \\
  {\vdots} \ar[d]^{N_0}\\ {J^{1,1}} \ar[d]^{N_0} \\ {J^{0,0}}}
\end{equation}
with $\dim J^{a,a}=n$ for $0\leq a\leq k$. We will denote this
structure by $\{J,N_0\}$.

First notice that any map $\phi\in \jgp_{-1}$ commuting with the
polarizing $N_0$ is completely determined by $\ti{\phi} =
\phi|_{J^{k,k}}$. Indeed, since for $0\leq a<k$,
$N_0|_{J^{a+1,a+1}}:J^{a+1,a+1}\rightarrow J^{a,a}$ is an isomorphism,
given $v^{a}\in J^{a,a}$ (for $0\leq a < k$) there exists $v^k\in
J^{k,k}$ such that $v^{a} = N_0^{k-a}(v^k)$. Therefore, $\phi(v^a) =
\phi(N_0^{k-a} (v^k)) = N_0^{k-a}(\phi(v^k))$.

The problem can be phrased in terms of matrices. In order to do that,
notice that $J^{k,k}$ is a pure Hodge structure of weight $2k$
polarized by $Q_k(\cdot,\cdot) = Q(\cdot ,N_0^k(\cdot))$ which is
symmetric nondegenerate and positive definite (on the real vector
space underlying $J^{k,k}$). Then, there is a $Q_k$-orthonormal basis
$\mathcal{B}_k=\{v^k_1,\ldots v^k_n\}$ of $J^{k,k}$. Using $N_0$ we
define $\mathcal{B}_a=\{v^a_1,\ldots,v^a_n\}$, where
$v^a_j=N_0^{k-a}(v^k_j)$ for $a=k-1,\ldots, 0$ and
$j=1,\ldots,n$. The set $\mathcal{B}=\cup_{a=0}^k \mathcal{B}_j$ is a
(real) basis of $\oplus_a J^{a,a}$. If
$[\ti{\phi}]_{\mathcal{B}_k,\mathcal{B}_{k-1}}$ is the matrix of
$\ti{\phi}$ with respect to the corresponding bases, the conditions
$\phi\in\lgc$ and commutativity with $N_0$ become that
$[\ti{\phi}]_{\mathcal{B}_k,\mathcal{B}_{k-1}} \in \C^{n\times n}$ is
symmetric, while commutativity of any two morphisms becomes
commutativity of the respective matrices in the standard way.

In this case, the maximal dimension problem for any $k$ and
$h^{a,k-a}= n$ for $0\leq a\leq k$ reduces to that of finding maximal
dimensional abelian subalgebras of symmetric matrices in $\gl(n,\C)$.

A first simplification comes from writing $\gl(n,\C) = \jgsl(n,\C)
\oplus \C$ and restricting to the $\jgsl$ part. The bound over $\gl$
will be $1$ unit higher and is realized, for example, by the direct
sum of a subalgebra maximizing dimension in $\jgsl$ and the linear
span of the identity matrix. Using the Cartan decomposition
$\jgsl(n,\R) = \jgk \oplus \jgp$ that can be interpreted as the
decomposition of trace zero matrices as the sum of antisymmetric and
symmetric matrices, we reduce the maximal dimension problem for
symmetric matrices to that of finding the maximal dimensional abelian
subalgebras of $\jgsl(n,\C)$ contained in $\jgp$. That maximal
dimension has been obtained by J. Carlson and D. Toledo using root
system techniques. They conclude in~\cite[\S 6]{ar:CT-rigidity} that
this maximal dimension (for subalgebras in $\jgp$) is
\begin{equation*}
  \begin{cases}
    \frac{1}{2}\alpha(\alpha+1) + \beta \text{ for } n = 2\alpha
    +\beta>1 \text{ and } \beta=0,1.\\
    0 \text{ for } n=1.
  \end{cases}
\end{equation*}
Furthermore, they show that for even $n$ all maximal dimensional
abelian subalgebras are conjugate while for odd $n$, there are two
conjugacy classes. All together we have the following result.

\begin{prop}\label{prop:max_dim_symmetric_matrices}
  Let $\{N_0;J\}$ be the nilpotent orbit of weight $k$
  of~\eqref{eq:PMHS-HT}, with $\dim J^{p,p} = n$ for $0\leq p\leq k$.
  Then, the maximal dimension of any \IVI $(\{N_0;J\},\ga)$ is
  \begin{equation*}
    \begin{cases}
      \frac{1}{2}\alpha(\alpha+1) + \beta + 1 \text{ for } n = 2\alpha
      +\beta>1 \text{ and } \beta=0,1.\\
      1 \text{ for } n=1.
    \end{cases}
  \end{equation*}
  Furthermore, up to conjugation, there is only one maximal
  dimension \IVI for $n$ even and two for $n$ odd.
\end{prop}

\begin{remark}
  It is easy to show either using Lie algebra theory or as a nice
  elementary computation~\cite[Theorem 5.5]{proc:CT-integralManifolds}
  that the maximal dimension of abelian subspaces of symmetric
  elements of $\gl(n,\R)$ is $n$. This result implies that the maximal
  dimension of the nilpotent cone for this MHS is $n$.
\end{remark}

We close with a comment regarding future work.  Carlson observed that
the maximal dimensions given by Theorem~\ref{th:CKTM} seem to be much
larger than the naturally occurring PVHS: for instance, hypersurface
variations are maximal PVHS of smaller
dimension~\cite{ar:CD-hypersurface_variations}. A question remains as
to what are the extra conditions that characterize the ``more
natural'' PVHS~\cite{ar:carlson-bounds_2}. It would be very
interesting to see if the \IVIs play a role in this respect since they
provide a finer classification than the IVHS and so could be linked to
specific degenerating behavior of the ``more natural'' PVHS.


\def\cprime{$'$}
\providecommand{\bysame}{\leavevmode\hbox to3em{\hrulefill}\thinspace}
\providecommand{\MR}{\relax\ifhmode\unskip\space\fi MR }
\providecommand{\MRhref}[2]{%
  \href{http://www.ams.org/mathscinet-getitem?mr=#1}{#2}
}
\providecommand{\href}[2]{#2}


\end{document}